\documentclass[10pt]{amsart}
\usepackage{enumerate,amsmath,amsthm,amssymb,cite}

\DeclareMathOperator{\real}{\mathrm{Re}}

\numberwithin{equation}{section}
\newtheorem{theorem}{Theorem}[section]

\theoremstyle{remark}
\newtheorem{remark}[theorem]{Remark}
\newtheorem{example}[theorem]{Example}
\newtheorem{definition}[theorem]{Definition}

\makeatletter
\@namedef{subjclassname@2010}{%
	\textup{2010} Mathematics Subject Classification}
\makeatother

\begin{document}

\title[Starlikeness Associated With The Exponential Function]{ Starlikeness Associated With The Exponential Function \boldmath}
\author[Adiba Naz]{Adiba Naz}

\address{Department of Mathematics,
	 University of Delhi,
	Delhi--110 007, India}
\email{adibanaz81@gmail.com}

\author[Sumit Nagpal]{Sumit Nagpal}

\address{Department of Mathematics, Ramanujan College, University of Delhi,	Delhi--110 019, India}
\email{sumitnagpal.du@gmail.com}

\author[V. Ravichandran]{V. Ravichandran}

\address{Department of Mathematics, National Institute of Technology,  Tiruchirappalli--620 015,  India}
\email{vravi68@gmail.com}

\begin{abstract}
Given a domain $\Omega$ in the complex plane $\mathbb{C}$ and a univalent function $q$ defined in an open unit disk $\mathbb{D}$ with nice boundary behaviour, Miller and Mocanu studied the class of admissible functions $\Psi(\Omega,q)$ so that the differential subordination $\psi(p(z),zp(z),z^2p''(z);z)\prec h(z)$ implies $p(z)\prec q(z)$  where $p$ is an analytic function in $\mathbb{D}$ with $p(0)=1$, $\psi:\mathbb{C}^3\times \mathbb{D}\to\mathbb{C}$ and $\Omega=h(\mathbb{D})$.  This paper investigates the properties of this class for $q(z)=e^z$. As application, several sufficient conditions for normalized analytic functions $f$  to be in the subclass of starlike functions associated with the exponential function are obtained.
\end{abstract}

\keywords{univalent functions, starlike functions, differential subordination, exponential function,  janowski starlike function}
\subjclass[2010]{30C45, 30C80}

\maketitle

\section{Introduction and Preliminaries }
Let $\mathcal{H}[a,n]$ denote the class of analytic functions defined in the open unit disk $\mathbb{D}:=\{z \in \mathbb{C}\colon|z|<1\}$ of the form   $f(z)= a + a_nz^n + a_{n+1}z^{n+1}+\cdots$, where $n$ is a positive integer and $a\in \mathbb{C}$. Set $\mathcal{H}_1:=\mathcal{H}[1,1]$. Let $\mathcal{H}$ be the subclass of $\mathcal{H}[0,1]$ consisting of functions $f$ normalized by the condition $f(0)=f'(0)-1=0$. Let $\mathcal{S}$ be a subclass of $\mathcal{H}$ containing univalent functions.  Given any two analytic functions in $\mathbb{D}$, we say that $f$ is subordinate to $g$, written as $f \prec g$, if there exists a Schwarz function $w$ that is analytic in $\mathbb{D}$ with $w(0)=0$ and $|w(z)|<1$
 satisfying $f(z)=g(w(z))$ for all $z\in\mathbb{D}$.
  In particular, if $g$ is univalent, then $f\prec g$ if and only if $f(0)=g(0)$ and $f(\mathbb{D})\subset g(\mathbb{D})$. Some special classes of univalent functions are of great significance in geometric function theory due to their geometric properties.
  By considering the analytic function $\varphi \in\mathcal{H}_1$ with positive real part in $\mathbb{D}$ that maps $\mathbb{D}$ onto regions which are starlike with respect to a point $\varphi(0)=1$ and symmetric with respect to the real axis, in 1994, Ma and Minda \cite{MR1343506}  gave a unified treatment of  various  subclasses of starlike functions in terms of subordination by studying the class
   \[\mathcal{S}^*(\varphi)= \left\{ f\in\mathcal{H}\colon \frac{zf'(z)}{f(z)} \prec \varphi(z),\,  z\in \mathbb{D} \right\}.   \]  For special choices of $\varphi$, the class $\mathcal{S}^*(\varphi)$ reduces to widely-known subclasses of starlike functions. For example, when $-1\leq B<A\leq 1$, $\mathcal{S}^*[A,B]:= \mathcal{S}^*((1+Az)/(1+Bz))$ is the class of Janowski \cite{MR0267103} starlike functions, $\mathcal{S}^*_P:=\mathcal{S}^*(1+2 (\log((1+\sqrt{z})/(1-\sqrt{z})))^2/\pi^2)$ is the class consisting of parabolic starlike functions \cite{MR1128729}, $\mathcal{S}^*_L:=\mathcal{S}^*(\sqrt{1+z})$ is the class of  lemniscate starlike functions \cite{sokol1996radius} and $\mathcal{S}^*_q:=\mathcal{S}^*(z+\sqrt{1+z^2})$ is the class of starlike functions associated with lune \cite{MR3469339}. In 2015, Mendiratta \textit{et al.\@} \cite{MR3394060} also introduced the class $\mathcal{S}^*_e=\mathcal{S}^*(e^z)$ of starlike functions associated with the  exponential function satisfying the condition $|\log (zf'(z)/f(z))|<1$ for $z\in\mathbb{D}$.



The study of differential subordination which is a generalized  form of differential inequalites began with a prodigious article ``Differential subordination and univalent functions'' by S.\@ Miller and P.\@ Mocanu \cite{MR616267} in 1981. After that the theory of differential subordination brought a revolutionary change and attracted many researchers to use this technique for the study of univalent functions. Given a complex function $\psi(r,s,t;z)\colon\mathbb{C}^3\times \mathbb{D}\to\mathbb{C}$ and a univalent function $h$ in $\mathbb{D}$, if $p$ is an analytic function in $\mathbb{D}$ that satisfies the \emph{second-order differential subordination} \begin{equation} \label{1.1}
	\psi(p(z), zp'(z), z^2p''(z);z) \prec h(z)
	\end{equation} then $p$ is called a \emph{solution} of the differential subordination. The univalent function $q$ is said to be a \emph{dominant} of the solutions of the differential subordination if $p\prec q$ for all $p$ satisfying \eqref{1.1}. A dominant $\tilde{q}$ that satisfies $\tilde{q} \prec q$ for all dominants $q$ of \eqref{1.1} is said to be the \emph{best dominant} of  \eqref{1.1}. The best dominant is unique upto a rotation of $\mathbb{D}$. Moreover let $\mathcal{Q}$ denote the set of analytic and univalent  functions $q$  in $\overline{\mathbb{D}} \setminus E(q)$, where \[ E(q) = \{\zeta \in \partial \mathbb{D} \colon \lim_{z\to \zeta} q(z)= \infty \}\] and are such that $q'(\zeta) \not=0$ for $\zeta \in \partial \mathbb{D} \setminus E(q)$. The following definition of admissible functions and the fundamental theorem laid the foundation stone in the theory of differential subordination.

\begin{definition}\cite[p.~27]{MR1760285}
  	Let $\Omega$ be a domain in $\mathbb{C}$, $q\in \mathcal{Q}$ and $n$ be a positive integer. Define $\Psi_n(\Omega,q)$ to be the \emph{class of admissible functions} $\psi\colon \mathbb{C}^3\times \mathbb{D} \to\mathbb{C}$ that satisfies the admissibility condition: \[\psi(r,s,t;z) \notin \Omega\] whenever \[r=q(\zeta) \text{ is finite, } s= m\zeta q'(\zeta) \text{ and }  \real \left(1+ \frac{t}{s}\right) \geq m \real \left( 1+\frac{\zeta q''(\zeta)}{q'(\zeta)} \right)\] where $z\in\mathbb{D}$, $\zeta \in \partial \mathbb{D}\setminus E(q)$ and $m \geq n$ is a positive integer. We write $\Psi_1(\Omega,q)$ as $\Psi(\Omega,q)$.
  \end{definition}
  \begin{theorem} \cite[p.~28]{MR1760285} \label{main}
  	Let $\psi \in \Psi_n(\Omega,q)$ with $q(0) =a$. If $p \in\mathcal{H}[a,n]$ satisfies \[\psi(p(z),zp'(z),z^2p''(z);z)\in\Omega \]
  then $p(z)\prec q(z)$.
  \end{theorem}
Miller and Mocanu \cite{MR1760285} in their monograph discussed the class of admissible functions $\Psi(\Omega,q)$ when the function $q$ maps $\mathbb{D}$ onto a disk or a half-plane. These two special classes together with Theorem \ref{main} lead to several important and interesting results in the theory of differential subordinations. However the aim of this paper is to consider differential implications with the superordinate function $q(z)=e^z$. In Section 2,  the admissibility class $\Psi(\Omega,e^z)$ is obtained, by deriving its admissibility condition. Examples are provided to illustrate the obtained results.

In 2015, Mendiratta \textit{et  al.\@} \cite{MR3394060} estimated bounds on $\beta$ for which $p(z)\prec e^z$ whenever $1 + \beta zp'(z)/p(z)$  is subordinate to $e^z$, $(1+Az)/(1+Bz)$ and $\sqrt{1+z}$. In 2018, Kumar and Ravichandran \cite{MR3800966} extended the result of Mendiratta \textit{et  al.\@} and obtained  bounds on $\beta$ for $1 + \beta zp'(z)/p^j(z)$ $(j=0,2)$.  They also estimated the bounds on $\beta$ such that $p(z)\prec e^z$ whenever $1 + \beta zp'(z)/p^j(z)$ $(j=0,1,2)$ is subordinate to $1+\sin z$ and $1+(z(k+z))/(k(k-z))$, where $k=\sqrt{2}+1$. Also, Gandhi \textit{et  al.\@}\cite{gandhi2018first} obtained bounds on $\beta$ for which $p(z)\prec e^z$ whenever $1 + \beta zp'(z)/p^j(z)$ $(j=0,1,2)$ is subordinate to $z+\sqrt{1+z^2}$. Motivated by their works and that of \cite{MR975653,MR2336133,MR3116529,MR3518217,MR2917253,MR1950727,MR2335444,MR2,MR704183}, in Section 3, the problem
\[\psi(p(z),zp(z),z^2p''(z);z)\prec h(z)\quad \Rightarrow \quad p(z)\prec e^z\]
is established for special cases of Janowski starlike functions $h$. In Section 4, the above said problem is solved for various expressions  when $h$ in particular, is also an exponential function. 
The results of \cite{MR3394060} are not only generalized but new differential implications are also obtained in last two sections.  Additionally, the applications of the results obtained yield sufficient conditions for functions $f\in\mathcal{H}$ to belong to the class $\mathcal{S}^*_e$.

\bigskip 	


\section{The Admissibility Condition }
In this section, we describe the admissible class $\Psi(\Omega,q)$ with examples, where $\Omega$ is a domain in $\mathbb{C}$ and $q(z)=e^z$. Note that $q$ is a univalent function in $\overline{\mathbb{D}}$ with $q(\mathbb{D}) = \Delta$ and $q(0)=1$, where $\Delta:=\{w \in \mathbb{C}\colon |\log w|<1\}$. Thus $q\in \mathcal{Q}$ with $E(q)=\emptyset$ and hence the class $\Psi(\Omega,q)$ is well-defined.

For $|\zeta| =1$, $q(\zeta) \in q(\partial \mathbb{D}) = \partial q(\mathbb{D}) = \{w \in \mathbb{C}:|\log w| =1\}$. This gives $|\log q(\zeta) | =1$ so that $\log q(\zeta) =e^{i\theta}$, where $\theta \in [0,2\pi)$ and hence $q(\zeta) = e^{e^{i\theta}}$. But  $q(\zeta) = e^\zeta$ which implies that $\zeta = e^{i\theta}$. Also $\zeta q'(\zeta) = e^{i\theta} e^{e^{i\theta}}$ and \begin{equation*}
 \real\left(1+\frac{\zeta q''(\zeta)}{q'(\zeta)}\right) = \real (1+e^{i\theta}) = 1+\cos \theta.
 \end{equation*}
 Thus the admissibility condition reduces to   \begin{equation} \label{adm} \begin{cases}  \begin{split}  \psi(r,s,t;z) \notin \Omega \quad  \text{whenever} \quad r = q(\zeta) = e^{e^{i\theta}} \\
 s = m\zeta q'(\zeta)  = m e^{i\theta} r \\
 \text{and} \quad \real (1+t/s) \geq m (1+\cos \theta)  \end{split} \end{cases}
 \end{equation}
where $z\in\mathbb{D}$, $\theta \in [0,2\pi)$ and $m \geq 1$. Therefore the class $\Psi(\Omega,e^z)$ consists of those functions $\psi\colon \mathbb{C}^3\times \mathbb{D}\to\mathbb{C}$ that satisfy the admissibility condition given by \eqref{adm}.  If $\psi\colon\mathbb{C}^2\times \mathbb{D}\to\mathbb{C}$, then the admissibility condition \eqref{adm} reduces to
 \[ \psi(e^{e^{i\theta}},m e^{i\theta} e^{e^{i\theta}};z)\not\in\Omega \] where $z\in\mathbb{D}$, $\theta \in [0,2\pi)$ and $m \geq 1$. As a particular case of Theorem \ref{main}, we have the following

 \begin{theorem} \label{thm A} 
 	Let $p \in\mathcal{H}_1$.
\begin{itemize}
 \item [(i)] If $\psi \in \Psi(\Omega,e^z)$, then
  \[\psi(p(z),zp'(z),z^2p''(z);z)\in\Omega\quad \Rightarrow\quad p(z)\prec e^z.\]
 \item [(ii)] If $\psi \in \Psi(\Delta,e^z)$, then
  \[\psi(p(z),zp'(z),z^2p''(z);z)\prec e^z\quad \Rightarrow\quad p(z)\prec e^z.\]
  \end{itemize}
\end{theorem}
\noindent We close this section with some examples illustrating Theorem \ref{thm A}.
\begin{example}
Let $\psi(r,s,t;z) = r+(1+2e)s$ and $h:\mathbb{D}\to\mathbb{C}$ be defined by
\[h(z)=2\left(\frac{2z+1}{2+z}\right).\]
Then $\Omega=h(\mathbb{D})=\{w\in \mathbb{C}:|w|<2\}$. To prove  $\psi \in \Psi(\Omega,e^z)$, we need to show that the admissibility condition \eqref{adm} is satisfied. Consider
\begin{align*}
|\psi(r,s,t;z)|&=e^{\cos\theta}|1+(1+2e)me^{i\theta}|\\
               &\geq e^{-1} ((1+2e)m-1)\geq 2
\end{align*}
whenever $r = e^{e^{i\theta}}$, $s = m e^{i\theta}r$ and $\real  \left(1+t/s\right)$ $\geq$ $m (1+\cos \theta )$, where $z\in \mathbb{D}$, $\theta \in [0,2\pi)$ and $m \geq 1$. Therefore $\psi(r,s,t;z)$ $\notin$ $\Omega$ and hence $\psi \in \Psi(\Omega,e^z)$. By Theorem \ref{thm A}, it follows that if $p\in\mathcal{H}_1$, then
\[|p(z)+(1+2e)zp'(z)|<2\quad \Rightarrow\quad p(z)\prec e^z.\]
\end{example}

\begin{example}
If $\psi(r,s,t;z) = 1+(1+\sqrt{2})es$ and $h(z)=\sqrt{1+z}$, then $\Omega=h(\mathbb{D})=\{w\in \mathbb{C}:|w^2-1|<1\}$. Consider
\begin{align*}
|(\psi(r,s,t;z))^2-1|&=(1+\sqrt{2})e|s||(1+\sqrt{2})es+2|\\
               &\geq (1+\sqrt{2})e|s|((1+\sqrt{2})e|s|-2)\\
               &=(1+\sqrt{2})me^{1+\cos \theta}((1+\sqrt{2})me^{1+\cos \theta}-2)\\
               &\geq (1+\sqrt{2})(\sqrt{2}-1)=1
\end{align*}
whenever $r = e^{e^{i\theta}}$, $s = m e^{i\theta}r$ and $\real  \left(1+t/s\right)$ $\geq$ $m (1+\cos \theta )$, where $z\in \mathbb{D}$, $\theta \in [0,2\pi)$ and $m \geq 1$. Thus $\psi(r,s,t;z)$ $\notin$ $\Omega$ and therefore $\psi \in \Psi(\Omega,e^z)$. By Theorem \ref{thm A}, it is easily seen that if $p\in\mathcal{H}_1$, then
\[|(1+(1+\sqrt{2})e zp'(z))^2-1|<1\quad \Rightarrow\quad p(z)\prec e^z.\]
\end{example}

  \begin{example}
Let $\psi(r,s,t;z) = 1+s$ and suppose that $\Omega= \{w \in\mathbb{C}\colon |w-1|<e^{-1}\}$. In order to prove $\psi \in \Psi(\Omega,e^z)$, note that \[|\psi(r,s,t;z)-1|=|s|= |me^{i\theta} e^{e^{i\theta}} | = me^{\cos\theta} \geq m e^{-1} \geq e^{-1}  \]
whenever $r = e^{e^{i\theta}}$, $s = m e^{i\theta}r$ and $\real  \left(1+t/s\right)$ $\geq$ $m (1+\cos \theta )$, where $z\in \mathbb{D}$, $\theta \in [0,2\pi)$ and $m \geq 1$. Therefore  $\psi(r,s,t;z)$ $\notin$ $\Omega$ which implies $\psi \in \Psi(\Omega,e^z)$. For any $p \in\mathcal{H}_1$, we obtain \[ \left| zp'(z)\right| < e^{-1} \quad\Rightarrow\quad  p(z)\prec e^z.\]
Similarly, if we take $\psi(r,s,t;z) =r^2-r+(1+e)s+1$ with the same $\Omega$ as defined earlier, then
\begin{align*}
|\psi(r,s,t;z)-1|&=|r^2-r+(2+e)s|\\
                 &=e^{\cos\theta}|e^{e^{i\theta}}-1+(2+e)me^{i\theta}|\\
                 &\geq e^{-1}((2+e)m-|e^{e^{i\theta}}-1|)\\
                 &\geq e^{-1}((2+e)m-1-e^{\cos\theta})\\
                 &\geq e^{-1}((2+e)m-1-e)\geq e^{-1}
\end{align*}
whenever $r = e^{e^{i\theta}}$, $s = m e^{i\theta}r$ and $\real  \left(1+t/s\right)$ $\geq$ $m (1+\cos \theta )$, where $z\in \mathbb{D}$, $\theta \in [0,2\pi)$ and $m \geq 1$. By Theorem \ref{thm A}, it is easy to deduce that for any $p \in\mathcal{H}_1$, we have
\[ |p^2(z)-p(z)+(1+e)zp'(z)| < e^{-1} \quad\Rightarrow\quad  p(z)\prec e^z.\]
In the similar fashion, by taking $\psi(r,s,t;z) = 1+s/r^2$ and $\Omega$ as above, it is easily seen that
 \[|\psi(r,s,t;z)-1|=|s/r^2| = |me^{i\theta}e^{-e^{i\theta}} | = me^{-\cos\theta} \geq me^{-1} \geq e^{-1}  \] whenever $r = e^{e^{i\theta}}$, $s = m e^{i\theta}r$ and $\real  \left(1+t/s\right)$ $\geq$ $m (1+\cos \theta )$, where $\theta \in [0,2\pi)$ and $m \geq 1$. This implies that $\psi(r,s,t;z)$ $\notin$ $\Omega$ and hence $\psi \in \Psi(\Omega,e^z)$. Thus for any $p \in\mathcal{H}_1$, we have \[ \left| \frac{zp'(z)}{p^2(z)}\right| < \frac{1}{e} \quad \Rightarrow \quad  |\log p(z)| < 1.\]
 \end{example}

\begin{example}
Let $\psi(r,s,t;z) = 1+s/r$ 
and $\Omega =h(\mathbb{D})=\{w \in\mathbb{C}\colon |w-1|<1\}$, where $h(z)=1+z$. Consider \[|\psi(r,s,t;z)-1|=|s/r| = |me^{i\theta} | = m \geq 1  \] whenever $r = e^{e^{i\theta}}$, $s = m e^{i\theta}r$ and $\real  \left(1+t/s\right)$ $\geq$ $m (1+\cos \theta )$, where $z\in \mathbb{D}$, $\theta \in [0,2\pi)$ and $m \geq 1$. Therefore $\psi(r,s,t;z)$ $\notin$ $\Omega$ and hence $\psi \in \Psi(\Omega,e^z)$. Using Theorem \ref{thm A}, in terms of subordination, the result can be written as
\[ 1+ \frac{zp'(z)}{p(z)} \prec 1+z\quad \Rightarrow \quad p(z) \prec e^z\]
where $p \in\mathcal{H}_1$. Since $\psi(q(z),zq'(z),z^2q''(z);z)= 1+z=h(z)$ and $\psi \in \Psi(\Omega,q)$, where $q(z)=e^z$, it follows that $e^z$ is the best dominant by \cite[Theorem 2.3e, p.~31]{MR1760285}.
 \end{example}

\begin{example}
Let $\psi(r,s,t;z) = 2s+t$ and $\Omega =\{w \in\mathbb{C}\colon |w|<1/e\}$. Then
\begin{align*}
|\psi(r,s,t;z)|&=|2s+t| = |s|\left|2+\frac{t}{s}\right|\\
                 &\geq me^{\cos \theta}\real\left(2+\frac{t}{s}\right)\\
                 &\geq me^{\cos \theta}(1+m (1+\cos \theta ))\\
                 &\geq me^{\cos \theta}\geq e^{-1}
\end{align*}
whenever $r = e^{e^{i\theta}}$, $s = m e^{i\theta}r$ and $\real  \left(1+t/s\right)$ $\geq$ $m (1+\cos \theta )$, where $z\in \mathbb{D}$, $\theta \in [0,2\pi)$ and $m \geq 1$. This shows that $\psi(r,s,t;z)$ $\notin$ $\Omega$ and $\psi \in \Psi(\Omega,e^z)$. By Theorem \ref{thm A}, the required result is
\[ |2zp'(z)+ z^2p''(z)|<1/e \quad \Rightarrow \quad p(z) \prec e^z.\]
\end{example}
\section{Subordination Associated with the Janowski Function}

For $-1\leq B<A\leq 1$, we  consider  the  subordination  $\psi(p(z),zp'(z),z^2p''(z);$ $z) \prec (1+Az)/(1+Bz)$ implying $p(z)\prec e^z$ for $z\in\mathbb{D}$. In particular, we first estimate the bound on $\beta$  such  that the first order differential subordination $1+\beta zp'(z)/p^n(z)\prec 1+(1-\alpha)z$ (where $n$ is any non-negative integer and $0\leq \alpha <1$) implies $p(z)\prec e^z$. Throughout this paper, we will  assume that $\beta$ is a positive real number and $r$, $s$, $t$ are  same as referred in the admissibility condition  \eqref{adm}.
\begin{theorem} \label{lem2}
If $n$ is a non-negative integer, $0\leq \alpha <1$ and $p\in \mathcal{H}_1$ satisfies the subordination \[1+\beta \frac{zp'(z)}{p^n(z)} \prec 1+(1-\alpha)z, \quad \text{where } \beta \geq  \begin{cases}
	e(1-\alpha) &\quad \text{when $n=0$}\\
	e^{n-1}(1-\alpha) &\quad \text{when $n\not= 0$}
	\end{cases}\]  then $p(z) \prec e^z$.
\end{theorem}
\begin{proof} Let $h(z)=1+(1-\alpha)z$, where $z\in \mathbb{D}$, $0\leq \alpha <1$ and $\Omega =h(\mathbb{D})=\{w \in\mathbb{C}\colon |w-1|<1-\alpha\}$.
	
	Case (i). If $n=0$, consider the function $\psi(r,s,t;z)= 1+\beta s$. Then
\[\psi(p(z),zp'(z),z^2p''(z);z)\prec 1+(1-\alpha)z.\]
Theorem \ref{thm A} is applicable if we show that $\psi\in \Psi(\Omega,e^z)$,  that is, $\psi(r,s,t;z)\not\in\Omega$ whenever $r = e^{e^{i\theta}}$, $s = m e^{i\theta}r$ and $\real  \left(1+t/s\right)$ $\geq$ $m (1+\cos \theta )$, where $z\in \mathbb{D}$, $\theta \in [0,2\pi)$ and $m \geq 1$. A simple calculation yields
\[|\psi(r,s,t;z)-1|=\beta m e^{\cos\theta}\geq \beta e^{-1}\geq 1-\alpha.\]
Hence $\psi(r,s,t;z)$ $\notin$ $\Omega$ which gives $\psi \in \Psi(\Omega,e^z)$. Using Theorem \ref{thm A}, we get $p(z)\prec e^z$.
	
	Case (ii). When $n\not= 0$, the function $\psi(r,s,t;z) = 1+\beta s/r^n$ satisfies
 \begin{equation*}
	|\psi(r,s,t;z)-1|= \beta m e^{-(n-1)\cos\theta} \geq \beta e^{-(n-1)}\geq1-\alpha
	\end{equation*}
	whenever $r = e^{e^{i\theta}}$, $s = m e^{i\theta}r$ and $\real  \left(1+t/s\right)$ $\geq$ $m (1+\cos \theta )$, where $z\in \mathbb{D}$, $\theta \in [0,2\pi)$ and $m \geq 1$. Therefore as argued in Case (i), $\psi(r,s,t;z)$ $\notin$ $\Omega$ which implies $\psi \in \Psi(\Omega,e^z)$. Hence by Theorem \ref{thm A}, we have the desired result.
\end{proof}
\begin{remark}
For the case $n=1$, Theorem \ref{lem2} reduces to  \cite[Theorem 2.8b, p.~376]{MR3394060} when $A=1-\alpha$ and $B=0$.
\end{remark}
Consequently, \textit{if a function $f\in\mathcal{H}$ satisfies the subordination \[1+\beta \left(\frac{zf'(z)}{f(z)}\right)^{1-n}\left(1-\frac{zf'(z)}{f(z)}+\frac{zf''(z) }{f'(z)} \right) \prec 1+(1-\alpha)z  \] where $0\leq \alpha <1$ and the bound on $\beta$ is defined as in Theorem \ref{lem2}, then $f \in \mathcal{S}^*_e$. }

\bigskip Next, the bound on $\beta$ is determined such that the first order differential subordination $1+\beta zp'(z)/p^{n+1}(z) \prec (2+z)/(2-z)$ (where $n$ is any non-negative integer) implies $p(z)\prec e^z$.
\begin{theorem} \label{lem3}
If $n$ is any non-negative integer and $p\in \mathcal{H}_1$ satisfies the subordination \[1+\beta\frac{zp'(z)}{p^{n+1}(z)} \prec \frac{2+z}{2-z}, \quad \text{where } \beta \geq 2e^{n}\]  then $p(z) \prec e^z$.
\end{theorem}
\begin{proof}
By considering the function $\psi(r,s,t;z)= 1+\beta s/r^{n+1}$ and $\Omega = \{w\in\mathbb{C} \colon |(2w-2)/(w+1) | <1 \}$, it suffices to show  $\psi\in\Psi(\Omega, e^z)$. For this, note that
 	\begin{equation*} 	\left|\frac{2\psi(r,s,t;z)-2}{\psi(r,s,t;z)+1}\right|
	 \geq \frac{2\beta m e^{-n\cos\theta} }{2+\beta m e^{-n\cos\theta}}. 	\end{equation*}
Since the real-valued function $g(x)=2x/(2+x)$ is increasing for $x\geq 0$ and  $\beta m e^{-n\cos\theta}\geq 2$, it is easy to deduce that
\begin{equation*} 	\left|\frac{2\psi(r,s,t;z)-2}{\psi(r,s,t;z)+1}\right|\geq 1 	\end{equation*}
whenever $r = e^{e^{i\theta}}$, $s = m e^{i\theta}r$ and $\real  \left(1+t/s\right)$ $\geq$ $m (1+\cos \theta )$, where $z\in \mathbb{D}$, $\theta \in [0,2\pi)$ and $m \geq 1$.
Hence by making use of Theorem \ref{thm A} we get the required result.
\end{proof}
\begin{remark}
The case $n=0$ in Theorem \ref{lem3} is similar to \cite[Theorem 2.8b, p.~376]{MR3394060} for $A=1/2$ and $B=-1/2$.
\end{remark}
As a result, we have

\textit{ If a function $f\in\mathcal{H}$  satisfies the subordination \[1+\beta \left(\frac{zf'(z)}{f(z)}\right)^{-n}\left(1-\frac{zf'(z)}{f(z)}+\frac{zf''(z) }{f'(z)} \right)  \prec \frac{2+z}{2-z} \]
	where   $\beta \geq 2e^n$ and $n$ is any non-negative integer, then $f \in \mathcal{S}^*_e$.}

\bigskip The next theorem provides a bound on $\alpha$ and  $\beta$ such that the first order differential subordination $(1-\alpha)p(z)+\alpha p^2(z)+\beta zp'(z)\prec 1+z$  implies $p(z)\prec e^z$.
\begin{theorem} \label{lem11}
	Let $\alpha$, $\beta$ be positive real numbers satisfying   $\alpha(e-1)+\beta e \geq e$ and  $p\in \mathcal{H}_1$. If  the following subordination
	\[(1-\alpha)p(z)+\alpha p^2(z)+\beta zp'(z)\prec 1+z \] holds, then $p(z)\prec e^z$.
\end{theorem}
\begin{proof} Let $\Omega = h(\mathbb{D}) =\{w\in\mathbb{C} \colon |w-1 | <1 \}$, where $h(z)=1+z$. If $\psi(r,s,t;z)  =(1-\alpha) r+\alpha r^2+\beta s$, the required subordination is proved if we show that $\psi\in \Psi(\Omega,e^z)$ in view of Theorem \ref{thm A}. Observe that
	\begin{align*}	|\psi(r,s,t;z)-1|^2
	 & = \big((1-\alpha)e^{\cos\theta} \cos(\sin\theta)+\alpha e^{2\cos\theta} \cos(2\sin\theta)+\beta m e^{\cos\theta} \cos \theta \cos(\sin\theta) \\ & \quad{}-\beta m e^{\cos\theta} \sin\theta \sin(\sin\theta)-1\big)^2  +\big((1-\alpha)e^{\cos\theta} \sin(\sin\theta)+\alpha e^{2\cos\theta} \sin(2\sin\theta) \\ & \quad{} +\beta m e^{\cos\theta} \cos \theta \sin(\sin\theta) +\beta m e^{\cos\theta} \sin\theta\cos(\sin\theta) \big)^2  =:g(\theta).
	\end{align*} 	
The second derivative test shows that the function $g$ attains its minimum value at $\theta=\pi$ for $\alpha>0$ and $\beta>0$. Therefore for all $\theta\in[0,2\pi)$
\begin{align*}
  g(\theta) &\geq g(\pi) = \frac{(\alpha (e-1) + e (e-1 + \beta m))^2}{e^4}  \\
   & \geq \frac{(\alpha (e-1) + e (e-1 + \beta))^2}{e^4}\geq 1
\end{align*}
by the given condition $\alpha (e-1)+\beta e \geq e$. Thus $\psi(r,s,t;z) \not\in\Omega$ whenever $r = e^{e^{i\theta}}$, $s = m e^{i\theta}r$ and $\real  \left(1+t/s\right)$ $\geq$ $m (1+\cos \theta )$, where $z\in \mathbb{D}$, $\theta \in [0,2\pi)$ and $m \geq 1$. By Definition \ref{1.1}, $\psi \in \Psi(\Omega,e^z)$ and the result is evident by Theorem \ref{thm A}.
\end{proof}
Thus \textit{if $\alpha (e-1)+\beta e \geq e$ and $f\in\mathcal{H}$   satisfies the following subordination
	\[ \left((1-\alpha)+(\alpha-\beta)\frac{zf'(z)}{f(z)}\right)\frac{zf'(z)}{f(z)}+\beta\frac{zf'(z)}{f(z)}\left(1+\frac{zf''(z) }{f'(z)} \right)  \prec 1+z \]
then $f \in \mathcal{S}^*_e$. }

\bigskip Next, we determine the bounds on $\beta$ such that the first order differential subordinations $p(z)+\beta zp'(z)/p^n(z)\prec (2+2z)/(2-z)$, where $n=0,1$  implies $p(z)\prec e^z$.
\begin{theorem} \label{lem12}
	Let $p\in \mathcal{H}_1$. Then both the following conditions are sufficient for $p(z)\prec e^z$:
	\begin{enumerate}[(a)]
		\item $p(z)+\beta zp'(z)\prec (2+2z)/(2-z)$ for $\beta\geq (e+2-\sqrt{2}(e-1))/(e(\sqrt{2}-1)) \approx 2.0323$.
		\item $p(z)+\beta zp'(z)/p(z)\prec (2+2z)/(2-z)$ for $\beta\geq (e+2-\sqrt{2}(e-1))/(\sqrt{2}-1) \approx 5.52436$.
	\end{enumerate}
\end{theorem}
\begin{proof}Define $h\colon\mathbb{D}\to\mathbb{C}$ by $h(z)= (2+2z)/(2-z)$ and suppose that  $\Omega=h(\mathbb{D})=\{w\in\mathbb{C}\colon |(w-1)/(w+2)|<1/2 \}$.
	
	(a) As done earlier in the previous results, the function $\psi(r,s,t;z)=r+\beta s$ should satisfies $\psi (r,s,t;z) \not\in\Omega$ whenever $r = e^{e^{i\theta}}$, $s = m e^{i\theta}r$ and $\real  \left(1+t/s\right)$ $\geq$ $m (1+\cos \theta )$, where $z\in \mathbb{D}$, $\theta \in [0,2\pi)$ and $m \geq 1$. Observe that \begin{equation*}
	\left|\frac{\psi(r,s,t;z)-1}{\psi(r,s,t;z)+2}\right|^2 = \frac{(1+\beta m \cos\theta-e^{-\cos\theta}\cos(\sin\theta))^2+(\beta m \sin\theta+e^{-\cos\theta}\sin(\sin\theta))^2}{(1+\beta m \cos\theta+2e^{-\cos\theta}\cos(\sin\theta))^2+(\beta m \sin\theta-2e^{-\cos\theta}\sin(\sin\theta))^2}.
	\end{equation*} It is easily verified that the minimum value of the function in the right hand side of the above equation occurs at $\theta=0$ and therefore we obtain \[ \left|\frac{\psi(r,s,t;z)-1}{\psi(r,s,t;z)+2}\right|^2 \geq \frac{(e-1+\beta e m)^2}{(e+2+\beta e m)^2} \geq \frac{(e(1+\beta)-1)^2}{(e(1+\beta)+2)^2}\geq \frac{1}{2} \]
since $\beta\geq (e+2-\sqrt{2}(e-1))/(e(\sqrt{2}-1))$.
	
(b) The required subordination is proved if we show that the function $\psi(r,s,t;z)=r+\beta s/r$ does not lie in $\Omega$. For $\beta\geq (e+2-\sqrt{2}(e-1))/(\sqrt{2}-1)$, using the same technique as in previous case, we have
\begin{align*}
\left|\frac{\psi(r,s,t;z)-1}{\psi(r,s,t;z)+2}\right|^2 & = \frac{(e^{\cos\theta}\cos(\sin\theta)+\beta m \cos\theta-1)^2+(e^{\cos\theta}\sin(\sin\theta)+\beta m \sin\theta)^2}{(e^{\cos\theta}\cos(\sin\theta)+\beta m \cos\theta+2)^2+(e^{\cos\theta}\sin(\sin\theta)+\beta m \sin\theta)^2} \\
& \geq\left( \frac{\beta m +e-1}{\beta m +e+2}\right)^2 \geq \frac{1}{2}
\end{align*} whenever $r = e^{e^{i\theta}}$, $s = m e^{i\theta}r$ and $\real  \left(1+t/s\right)$ $\geq$ $m (1+\cos \theta )$, where $z\in \mathbb{D}$, $\theta \in [0,2\pi)$ and $m \geq 1$. Therefore using Theorem \ref{thm A}, we have $p(z)\prec e^z$.
\end{proof}
As a consequence, we obtain

\textit{If a function $f\in\mathcal{H}$ satisfies either of the following subordinations \begin{enumerate}[(a)]
		\item
		$\dfrac{zf'(z)}{f(z)}+\beta\dfrac{zf'(z)}{f(z)} \left(1-\dfrac{zf'(z)}{f(z)} +\dfrac{zf''(z)}{f'(z)}\right)\prec \dfrac{2+2z}{2-z}$
		for $\beta \geq \dfrac{e+2-\sqrt{2}(e-1)}{e(\sqrt{2}-1)}$
		\item 	$\dfrac{zf'(z)}{f(z)}+\beta \left(1-\dfrac{zf'(z)}{f(z)} +\dfrac{zf''(z)}{f'(z)}\right) \prec \dfrac{2+2z}{2-z}$
		for $\beta \geq  \dfrac{e+2-\sqrt{2}(e-1)}{\sqrt{2}-1}$
	\end{enumerate} then $f \in \mathcal{S}^*_e$. }

\section{Subordination Associated With The Exponential Function}
In this section, we consider the problem of determining the conditions under which the subordination $\psi(p(z),zp'(z),$ $z^2p''(z);z) \prec e^z$ implies that $p(z)\prec e^z$ also holds. Alternatively, our aim is to show that $\psi\in\Psi\{e^z\}:=\Psi(\Delta, e^z)$ for various choices of $\psi$, where $\Delta:=\{w\in\mathbb{C}\colon |\log w|<1 \}$. The first theorem of this section estimates the bound on $\beta$ such that  the first order differential subordination $1+\beta (zp'(z))^n \prec e^z$ (where $n$ is any positive integer) implies $p(z)\prec e^z$. Recall that, for $z\not=0$ \[\log z = \ln|z|+i\arg z= \ln ( x^2+y^2)^{1/2}+ i \tan^{-1}(y/x) \quad \text{for } x>0. \]
\begin{theorem} \label{lem1}
	If $n$ is any positive integer and $p\in\mathcal{H}_1$ satisfies the subordination \[1+\beta (zp'(z))^n \prec e^z, \quad \text{where } \beta \geq  \begin{cases}
	e^{n+1}+e^{n} \quad \text{when $n$ is odd}\\
	e^{n+1}-e^n \quad \text{when $n$ is even}
	\end{cases}\] then $p(z) \prec e^z$.
\end{theorem}
\begin{proof}
 The required subordination is proved if we show that the function defined as $\psi(r,s,t;z)=1+\beta s^n$ satisfies the condition $\psi(r,s,t;z)\not\in \Omega$ whenever $r = e^{e^{i\theta}}$, $s = m e^{i\theta}r$ and $\real  \left(1+t/s\right)$ $\geq$ $m (1+\cos \theta )$, where $z\in \mathbb{D}$, $\theta \in [0,2\pi)$ and $m \geq 1$. Consider
	 \[|\log \psi(r,s,t;z)|^2 =\frac{1}{4} \ln^2 (\mu^2+\nu^2) + \left(\tan^{-1}\frac{\mu}{\nu}\right)^2=: g(\theta)\]	where \[\mu= \beta m^n e^{n\cos\theta}\sin n\theta\cos(n\sin\theta)  +\beta m^n e^{n\cos\theta}\cos n\theta\sin(n\sin\theta)\] and\[ \nu =1+\beta m^n e^{n\cos\theta}\cos n\theta\cos(n\sin\theta) -\beta m^n e^{n\cos\theta}\sin n\theta\sin(n\sin\theta). \]
	
	Case (i). When $n$ is odd and $\beta \geq e^{n+1}+e^n$, we have\[g''(\pi) =\frac{\beta m^n n\ln\left(e^{-2n}(-e^n+\beta m^n)^2\right)}{-e^n+\beta m^n} >0\]  for all $m\geq 1$. Therefore second derivative verifies that minimum value of $g$ is attained at $\theta =\pi$. If $\beta \geq e^{n+1}+e^n$, we obtain \[g(\theta) \geq g(\pi) =\frac{1}{4}\ln^2\left(1-\frac{2\beta m^n}{e^n}+\frac{\beta^2 m^{2n}}{e^{2n}}\right)\geq \frac{1}{4} \ln^2 \left(1-\frac{\beta m^n}{e^n} \right)^2 \geq 1  \] for all $\theta\in[0,2\pi)$. Thus $|\log \psi(r,s,t;z)| \geq 1$ and Theorem \ref{thm A} gives $\psi \in \Psi\{e^z\}$.
	
	Case (ii). If $n$ is even and \[g''(\pi) =\frac{\beta m^n n\ln\left(e^{-2n}(e^n+\beta m^n)^2\right)}{e^n+\beta m^n} >0\] for   $\beta>0$, the minimum value of the function $g$ is attained at $\theta =\pi$. Therefore for all $\theta\in[0,2\pi)$ and $\beta\geq e^{n+1}-e^n$, we get \[g(\theta) \geq g(\pi) =\frac{1}{4}\ln^2\left(1+\frac{2\beta m^n}{e^n}+\frac{\beta^2 m^{2n}}{e^{2n}}\right) \geq 1. \]This implies that $\psi \in \Psi\{e^z\}$. Hence Theorem \ref{thm A} gives the desired differential subordination.
\end{proof}
Now, we estimate the bound on $\beta$ such that the first order differential subordination $1+\beta zp'(z)/p^{n+1}(z)\prec e^z$ (where $n$ is any non-negative integer) implies $p(z)\prec e^z$.
\begin{theorem} \label{lem4}
If $p\in\mathcal{H}_1$ satisfies the subordination \[1+\beta\frac{zp'(z)}{p^{n+1}(z)} \prec e^z, \quad \text{where } \beta \geq e^{n+1}-e^n\] and $n$ is any non-negative integer, then $p(z) \prec e^z$.
\end{theorem}
\begin{proof} We apply Theorem \ref{thm A} to show that $\psi\in \Psi\{e^z\}$, where $\psi(r,s,t;z)= 1+\beta s/r^n$. Whenever $r = e^{e^{i\theta}}$, $s = m e^{i\theta}r$ and $\real  \left(1+t/s\right)$ $\geq$ $m (1+\cos \theta )$, where $z\in \mathbb{D}$, $\theta \in [0,2\pi)$ and $m \geq 1$, note that
\[	|\log \psi(r,s,t;z)|^2
	 = \frac{1}{4}\ln^2 (\mu^2+\nu^2)+ \left(\tan^{-1}\frac{\mu}{\nu}\right)^2=: g(\theta)\]
where
\[\mu=\beta m e^{-n\cos\theta}\sin\theta\cos(n\sin\theta) -\beta m e^{-n\cos\theta}\cos\theta\sin(n\sin\theta)\] and \[\nu= 1+\beta m e^{-n\cos\theta}\cos\theta\cos(n\sin\theta)+\beta m e^{-n\cos\theta}\sin\theta\sin(n\sin\theta).\]
	Let $u(x)=x(-1+n)^2-\left(-xn+e^n(1-3n+n^2)\right)\ln \left(e^{-n}(e^n+x)\right)$, where $x>0$ and $n$ is a non-negative integer. Natural logarithm being an increasing function implies that  $\ln(e^n+x)>\ln(e^n)$ for $x>0$, that is, $\ln(e^n+x)>n$ for $x>0$. This gives \[u(x)>x(1-2n)+n\left(xn-e^n(1-3n+n^2)\right)+ne^n(1-3n+n^2)=x(1-n)^2>0\] for $x>0$ and $n\not=1$. In particular for $n=1$, $u(x)=(x+e)(\ln(x+e)-1)>0$ for $x>0$.
	Therefore  \[g''(0) = \frac{\beta m \left(2 \beta m (-1 + n)^2 -\left(-\beta m n + e^n (1 - 3 n + n^2)\right) \ln\left(e^{-2n}(e^n+\beta m)^2\right)\right)}{(e^n+\beta m)^2}>0\] for $\beta>0$, which implies, using second derivative test, $g$ attains its minimum value at $\theta=0$.
	Hence for all $\theta\in[0,2\pi)$ and $\beta \geq e^{n+1}-e^{n}$
 \[g(\theta) \geq g(0) =\frac{1}{4}\ln^2\left(1+\frac{2\beta m}{e^n}+\frac{\beta^2 m^2}{e^{2n}}\right)=\frac{1}{4}\ln^2\left(1+\frac{\beta m}{e^n}\right)^2 \geq \frac{1}{4}\ln^2\left(1+\frac{e^{n+1}-e^{n}}{e^n}\right)^2 =1.\]
  Thus $|\log \psi(r,s,t;z)| \geq 1$ which implies $\psi \in \Psi\{e^z\}$.   
\end{proof}
\begin{remark}
	For $n=0$, Theorem \ref{lem4} reduces to \cite[Theorem 2.8a, p.~376]{MR3394060}.
\end{remark}
In the next two theorems, the bound on $\beta$ is computed such that the first order differential subordination $p(z)+\beta zp'(z)/p^{n+1}$ $(z)\prec e^z$ (where $n=-1, 0,1,2,\ldots$) implies $p(z)\prec e^z$.
\begin{theorem} \label{lem5}
	Let $p\in \mathcal{H}_1$, then each of the following subordinations are sufficient for $p(z)\prec e^z$:
	\begin{enumerate}[(a)]
		\item $p(z)+\beta zp'(z) \prec e^z$ for $\beta \geq e^2+1 \approx 8.38906$.
		\item $p(z)+\beta zp'(z)/p(z) \prec e^z$ for $\beta \geq e+e^{-1} \approx 3.08616$.
	\end{enumerate}
\end{theorem}
\begin{proof} (a) In order to prove the admissibility condition \eqref{adm} for the function $\psi(r,s,t;z)= r+\beta s$, we need to show that $|\log \psi(r,s,t;z)|^2\geq 1$ whenever $r = e^{e^{i\theta}}$, $s = m e^{i\theta}r$ and $\real  \left(1+t/s\right)$ $\geq$ $m (1+\cos \theta )$, where $z\in \mathbb{D}$, $\theta \in [0,2\pi)$ and $m \geq 1$.  A simple computation gives
\begin{align*} 	|\log \psi(r,s,t;z)|^2
	& =\frac{1}{4}\ln^2 (e^{2\cos\theta} +\beta^2 m^2 e^{2\cos\theta} +2\beta m e^{2\cos\theta}\cos\theta) \\ & \quad  + \left(\tan^{-1}\left(\frac{\sin(\sin\theta)+\beta m \cos\theta\sin(\sin\theta)+\beta m \sin\theta\cos(\sin\theta)}{\cos(\sin\theta)+\beta m \cos\theta\cos(\sin\theta)-\beta m \sin\theta\sin(\sin\theta)}\right)\right)^2=:g(\theta).
	\end{align*}
Note that \[g''(\pi) =\frac{2\beta m(1-\beta m) +(1-\beta m+\beta^2m^2)\ln((-1+\beta m)^2)}{(-1+\beta m)^2} >0\] for $\beta >\beta^*\approx 3.4446$, where $\beta^*$ is the root of the equation $x (1-x) + (1 - x + x^2)\ln(-1 + x)=0$. Therefore the minimum value of the function $g$ is clearly attained at $\theta =\pi$ for $\beta \geq e^2+1 \approx 8.38906$. In that case, we have \[g(\theta) \geq g(\pi) =\frac{1}{4}\ln^2\left(\frac{1}{e^2}-\frac{2\beta m}{e^2}+\frac{\beta^2 m^2}{e^2}\right) \geq 1 \quad\text{for all }\theta\in[0,2\pi).\] Hence $\psi \in \Psi\{e^z\}$.  
	
	(b) Using the same technique as above, for the function $\psi(r,s,t;z)= r+\beta s/r$, consider
	\begin{align*} |\log \psi(r,s,t;z)|^2
	& = \frac{1}{4}\ln^2 (e^{2\cos\theta} +\beta^2 m^2+ 2\beta me ^{\cos\theta} \cos\theta \cos(\sin\theta) +2\beta m e^{\cos\theta}\sin\theta\sin(\sin\theta)) \\ & \quad  + \left(\tan^{-1}\frac{e^{\cos\theta} \sin(\sin\theta)+\beta m \sin \theta}{e^{\cos\theta} \cos(\sin\theta)+\beta m \cos \theta}\right)^2 =:g(\theta).
	\end{align*}
We observe that the second derivative of  $g$ is positive on both of its critical points, therefore 
 the  absolute minimum of $g$ is attained at $\theta=\pi$ for $\beta\geq e+e^{-1}$. Hence we get \[g(\theta) \geq g(\pi) =\frac{1}{4}\ln^2\left(\frac{1}{e^2}-\frac{2\beta m}{e}+\beta^2 m^2\right)\geq 1 \quad \text{for all }\theta\in[0,2\pi).\] Thus $\psi\in\Psi\{e^z\}$ and Theorem \ref{thm A} completes the proof. 
\end{proof}
\begin{theorem} \label{lem6}
	Let $n$ be any positive integer and $\beta_n$ be a positive root of the equation
	\begin{equation}\label{root}
	  \left(e^{1 + n} - x (-1 + n)\right)^2- \left(e^{2+2 n} n - x^2 n + 	x e^{1 + n} (1 - n + n^2)\right)  \ln(e+e^{-n}x) =0.
	\end{equation} If $p\in\mathcal{H}_1$  satisfies the subordination \[p(z)+\beta\frac{zp'(z)}{p^{n+1}(z)} \prec e^z, \quad \text{for } \beta >\beta_n \] then $p(z) \prec e^z$.
\end{theorem}
\begin{proof}
As argued in other cases, to prove the required subordinaton, it suffices to show that the function $\psi(r,s,t;z)=r+\beta s/r^{n+1}$ satisfies $\psi(r,s,t;z)\not\in\Delta$  whenever $r = e^{e^{i\theta}}$, $s = m e^{i\theta}r$ and $\real  \left(1+t/s\right)$ $\geq$ $m (1+\cos \theta )$, where $z\in \mathbb{D}$, $\theta \in [0,2\pi)$ and $m \geq 1$.  Note that
	 \begin{align*}|\log \, \psi(r,s,t;z)|^2 & = \frac{1}{4}\ln^2 (e^{2\cos\theta} +\beta^2 m^2 e^{-2n\cos\theta} +2\beta m e^{(1-n)\cos\theta}\cos\theta\cos(\sin\theta)\cos(n\sin\theta)\\ & \quad  +2\beta m e^{(1-n)\cos\theta}\sin\theta\cos(\sin\theta)\sin(n\sin\theta) +2\beta m e^{(1-n)\cos\theta}\sin\theta\sin(\sin\theta) \\ & \quad \cos(n\sin\theta)  -2\beta m e^{(1-n)\cos\theta}\cos\theta\sin(\sin\theta)\sin(n\sin\theta)) + \left(\tan^{-1}(\chi)\right)^2 =:g(\theta)
	\end{align*} where \[\chi=\frac{e^{(n+1)\cos\theta}\sin(\sin\theta)  + \beta m  \sin\theta\cos(n\sin\theta) -\beta m  \cos\theta\sin(n\sin\theta)}{e^{(n+1)\cos\theta}\cos(\sin\theta)  +\beta m \cos\theta\cos(n\sin\theta)  +\beta m \sin\theta\sin(n\sin\theta)}. \]
	If $\beta>\beta_n$, where $\beta_n$ is a positive root of the equation \eqref{root},  then  \begin{gather*} g''(0) =\frac{2 \left(e^{1 + n} - \beta m (-1 + n)\right)^2 - \left(e^{2+2 n} n - \beta^2 m^2 n + 	\beta  e^{1 + n}m (1 - n + n^2)\right)  \ln\left((e+\beta e^{-n}m)^2\right)}{(e^{1+n}+\beta m)^2}.  >0\end{gather*}  Therefore the minimum value of $g$ is attained at $\theta =0$ by the second derivative test which implies for all $\theta\in[0,2\pi)$ and $\beta>\beta_n>0$ \[g(\theta) \geq g(0) =\frac{1}{4}\ln^2\left(e^{2} + \frac{2\beta m}{e^{n-1}}+\frac{\beta^2 m^2}{e^{2n}}\right) \geq \frac{1}{4}\ln^2(e^{2}) \geq \frac{1}{4}(2)^2 =1 \]  for all positive integers $m$ and $n$. Hence $\psi\in\Psi\{e^z\}$ and Theorem \ref{thm A} gives the desired result.
\end{proof}
Next the bound on $\beta$ is determined such that each of the first order differential subordination $p(z)+\beta (zp'(z))^2/p^n(z)\prec e^z$ ($n =0,1,2$) implies $p(z)\prec e^z$.
\begin{theorem} \label{lem7}
	Let $p \in\mathcal{H}_1$. Then each of the following subordinations are sufficient for $p(z)\prec e^z$:
	\begin{enumerate}[(a)]
		\item $p(z)+\beta (zp'(z))^2 \prec e^z$ for $\beta \geq e^3-e \approx 17.3673$.
		\item $p(z)+\beta (zp'(z))^2/p(z) \prec e^z$ for $\beta \geq e^2-1 \approx 6.38906$.
		\item  $p(z)+\beta (zp'(z))^2/p^2(z) \prec e^z$ for $\beta \geq e-e^{-1} \approx 2.3504$.
	\end{enumerate}
\end{theorem}
\begin{proof} For different choices of $\psi$, we need to show that $\psi\in \Psi\{e^z\}$, that is, we must verify $|\log\psi(r,s,t;z)| \geq 1$ whenever $r = e^{e^{i\theta}}$, $s = m e^{i\theta}r$ and $\real  \left(1+t/s\right)$ $\geq$ $m (1+\cos \theta )$, where $z\in \mathbb{D}$, $\theta \in [0,2\pi)$ and $m \geq 1$.
	
	(a) Let $\psi(r,s,t;z)=r+\beta s^2$ and consider
	 \begin{equation*} 	|\log \,   \psi(r,s,t;z)|^2 = \frac{1}{4}\ln^2 (\mu^2+\nu^2)  + \left(\tan^{-1}\frac{\mu}{\nu}\right)^2  =: g(\theta)
	\end{equation*}
	where \[\mu=e^{\cos\theta}\sin(\sin\theta)+\beta m^2 e^{2\cos\theta}\cos 2\theta\sin(2\sin\theta)+\beta m^2 e^{2\cos\theta}\sin 2\theta\cos(2\sin\theta)\]and \[\nu=e^{\cos\theta}\cos(\sin\theta)+\beta m^2e^{2\cos\theta} \cos 2\theta\cos(2\sin\theta)-\beta m^2 e^{2\cos\theta}\sin 2\theta\sin(2\sin\theta). \]
 Since  \[g''(\pi)= \frac{-2 (e + 2 \beta m^2)^2 + (e^2 + 2 \beta e m^2 +
		2 \beta^2 m^4) \ln\left((e + \beta m^2)^2\right)}{(e + \beta m^2)^2} >0\]  for $\beta >\beta^* \approx 3.7586$, where $\beta^*$ is a positive root of the equation $(e + 2 x)^2 - (e^2 + 2 x e + 2 x^2) \ln(e + x)=0$, we can say that the minimum value of $g$ is obviously attained at $\theta =\pi$ for $\beta \geq e^3-e$. Therefore for $\beta \geq e^3-e$, we have \[g(\theta) \geq g(\pi) =\frac{1}{4}\ln^2\left(\frac{1}{e^2}+\frac{2\beta m^2}{e^3}+\frac{\beta^2 m^4}{e^4}\right) \geq 1\quad \text{for all }\theta\in[0,2\pi).\]  Hence $ \psi(r,s,t;z)\not\in\Delta$ and using Theorem \ref{thm A} the result follows. 
	
		(b)	Let the function be defined by $\psi(r,s,t;z)= r+\beta s^2/r$ and observe
	 \begin{align*}
	 |\log \psi(r,s,t;z)|^2
	= \frac{1}{4}\ln^2 (e^{2\cos\theta} +\beta^2 m^4 e^{2\cos\theta} +2\beta m^2 e^{2\cos \theta}\cos 2\theta)  + \left(\tan^{-1}(\chi)\right)^2 =:g(\theta)
	\end{align*}
	where \[\chi=\frac{\sin(\sin\theta)+\beta m^2 \cos 2\theta\sin(\sin\theta) +\beta m^2 \sin 2\theta\cos(\sin\theta)}{\cos(\sin\theta)  +\beta m^2 \cos 2\theta\cos(\sin\theta)-\beta m^2 \sin 2\theta\sin(\sin\theta)}. \]
	It is  easily verified that the minimum value of the function $g$ is attained at $\theta=\pi$ for $\beta\geq e^2-1$. In that case for all $\theta\in[0,2\pi)$ and $\beta\geq e^2-1$ \[g(\theta) \geq g(\pi) =\frac{1}{4}\ln^2\left(\frac{1}{e^2}+\frac{2\beta m^2}{e^2}+\frac{\beta^2 m^4}{e^2}\right) \geq 1. \] Therefore $|\log \psi(r,s,t;z)| \geq 1$ whenever $r = e^{e^{i\theta}}$, $s = m e^{i\theta}r$ and $\real  \left(1+t/s\right)$ $\geq$ $m (1+\cos \theta )$, where $z\in \mathbb{D}$, $\theta \in [0,2\pi)$ and $m \geq 1$. Hence Theorem \ref{thm A} yields the desired result. 

	(c) For the function $\psi(r,s,t;z)=r+\beta s^2/r^2$, it is easy to deduce that
	\begin{align*} |\log \psi(r,s,t;z)|^2
	& = \frac{1}{4}\ln^2 (e^{2\cos\theta} +\beta^2 m^4+ 2\beta m^2e ^{\cos\theta} \cos 2\theta \cos(\sin\theta) +2\beta m^2 e^{\cos\theta}\sin 2\theta \\ & \quad \sin(\sin\theta))  + \left(\tan^{-1}\left(\frac{e^{\cos\theta} \sin(\sin\theta)+\beta m^2 \sin 2\theta}{e^{\cos\theta} \cos(\sin\theta)+\beta m^2 \cos 2\theta}\right)\right)^2 =:g(\theta).
	\end{align*}
	Since the second derivative of $g$ is positive on both of its critical points,  $g$ attains absolute minimum at $\theta=\pi$ for $\beta\geq e-e^{-1}>0$. Therefore for all $\theta\in[0,2\pi)$ and for $\beta\geq e-e^{-1}$, 	we get  \[g(\theta) \geq g(\pi) =\frac{1}{4}\ln^2\left(\frac{1}{e^2}+\frac{2\beta m^2}{e}+\beta^2 m^4\right) \geq 1.\]  Hence  $\psi(r,s,t;z)\notin \Delta$ and thus $\psi\in\Psi\{e^z\}$.
\end{proof}
Next, we estimate the bound on $\beta$ such that each of the first order differential subordination $p^2(z)+\beta zp'(z)/p^n(z)\prec e^z$ ($n =0,1,2$) implies $p(z)\prec e^z$.
\begin{theorem} \label{lem8}
	Let $p\in \mathcal{H}_1$. Then each of the following subordinations are sufficient for $p(z)\prec e^z$:
	\begin{enumerate}[(a)]
		\item $p^2(z)+\beta zp'(z)\prec e^z$ for $\beta \geq e^2+e^{-1} \approx 7.75694$.
		\item $p^2(z)+\beta zp'(z)/p(z) \prec e^z$ for $\beta \geq e+e^{-2} \approx 2.85362$.
		\item  $p^2(z)+\beta zp'(z)/p^2(z) \prec e^z$ for $\beta > \beta^*\approx 104.122$, where $\beta^*$ is a positive root of the equation \begin{equation*}6e^6+5x e^3 -x^2+(-2e^6-5xe^3+x^2) \ln(e^3 + x)=0. \end{equation*}
	\end{enumerate}
\end{theorem}
\begin{proof} (a) For the function $\psi(r,s,t;z)=r^2+\beta s$, Theorem \ref{thm A} is applicable if we show that the function $\psi\in\Psi\{e^z\}$  whenever $r = e^{e^{i\theta}}$, $s = m e^{i\theta}r$ and $\real  \left(1+t/s\right)$ $\geq$ $m (1+\cos \theta )$, where $z\in \mathbb{D}$, $\theta \in [0,2\pi)$ and $m \geq 1$. Consider
\[|\log \, \psi(r,s,t;z)|^2
	=\frac{1}{4}\ln^2 (\mu^{2}+\nu^2) + \left(\tan^{-1}\frac{\mu}{\nu}\right)^2  =:g(\theta)\]
	where \[\mu=e^{2\cos\theta}\sin(2\sin\theta)+\beta m e^{\cos\theta}\cos \theta\sin(\sin\theta)  +\beta m e^{\cos\theta}\sin \theta\cos(\sin\theta)\] and\[\nu=e^{2\cos\theta}\cos(2 \sin\theta) +\beta m e^{\cos\theta} \cos \theta\cos(\sin\theta) -\beta m e^{\cos\theta}\sin \theta\sin(\sin\theta). \]
	To show $|\log \psi(r,s,t;z)|\geq 1$ note that $g''(\pi)>0$ for  $\beta >\beta^* \approx 2.9432$, where $\beta^*$ is a  root of the equation $2xe(1+xe)-(2+xe+x^2e^2)\ln(-1+xe)=0$. Therefore minimum value of the function $g$ is obviously attained at $\theta =\pi$ for $\beta \geq e^2+e^{-1}$ and hence \[g(\theta) \geq g(\pi) =\frac{1}{4}\ln^2\left(\frac{1}{e^4}-\frac{2\beta m}{e^3}+\frac{\beta^2 m^2}{e^2}\right) \geq 1\quad\text{for all }\theta\in[0,2\pi).\] Consequently, Theorem \ref{thm A}  yields the result.
	
	(b) The required subordination is proved if we show that $\psi\in \Psi\{e^z\}$, that is, if the admissibility condition \eqref{adm} is satisfied. For the function $\psi(r,s,t;z)= r^2+\beta s/r$, observe
	\begin{align*} |\log \psi(r,s,t;z)|^2
	& = \frac{1}{4}\ln^2 (e^{4\cos\theta} +\beta^2 m^2+ 2\beta m e ^{2\cos\theta} \cos \theta \cos(2\sin\theta) +2\beta m e^{2\cos\theta}\sin \theta  \\ & \quad  \sin(2\sin\theta)) + \left(\tan^{-1}\left(\frac{e^{2\cos\theta} \sin(2\sin\theta) + \beta m \sin \theta}{e^{2\cos\theta} \cos(2\sin\theta)+\beta m \cos \theta}\right)\right)^2 =:g(\theta).
	\end{align*}
	For $\beta\geq e+e^{-2}$, it is easily verified using second derivative test that $g$ attains its minimum value at $\theta=\pi$  which implies \[g(\theta) \geq g(\pi) =\frac{1}{4}\ln^2\left(\frac{1}{e^4}-\frac{2\beta m}{e^2}+\beta^2 m^2\right)\geq 1 \quad \text{for all }\theta\in[0,2\pi).\]  Therefore  $\psi(r,s,t;z)\notin \Delta$  whenever $r = e^{e^{i\theta}}$, $s = m e^{i\theta}r$ and $\real  \left(1+t/s\right)$ $\geq$ $m (1+\cos \theta )$, where $z\in \mathbb{D}$, $\theta \in [0,2\pi)$ and $m \geq 1$. Using Theorem \ref{thm A} we get $p(z)\prec e^z$.
	
	(c)  As done in other cases, we need to show that $\psi(r,s,t;z)\notin \Delta$, where $\psi$ is defined as $\psi(r,s,t;z)=r^2+\beta s/r^2$. Consider
	 \begin{equation*} |\log  \psi(r,s,t;z)|^2
	 = \frac{1}{4}\ln^2 (\mu^2+\nu^2)  + \left(\tan^{-1}\frac{\mu}{\nu}\right)^2  =:g(\theta)
	\end{equation*} where \[\mu=e^{2\cos\theta}\sin(2\sin\theta)-\beta m e^{-\cos\theta}\cos \theta\sin(\sin\theta)+\beta m e^{-\cos\theta}\sin \theta\cos(\sin\theta)\] and \[\nu=e^{2\cos\theta}\cos(2\sin\theta)+\beta m e^{-\cos\theta}\cos \theta\cos(\sin\theta) +\beta m e^{-\cos\theta}\sin \theta\sin(\sin\theta). \]
	We note that the minimum value of $g$ is attained at $\theta=0$   for $\beta >\beta^* \approx 104.122$, where $\beta^*$ is a positive root of the equation $ 6e^6+5x e^3 -x^2+(-2e^6-5xe^3+x^2) \ln(e^3 + x)=0$.  For $\theta\in[0,2\pi)$ and $\beta>\beta^*$, we have \[g(\theta) \geq g(0) =\frac{1}{4}\ln^2\left(e^4+2\beta me+\frac{\beta^2 m^2}{e^2}\right) \geq 1.\]  Therefore $\psi\in\Psi\{e^z\}$ and hence the result.
\end{proof}

Next, the bound on $\beta$ is ascertained such that each of the first order differential subordination $p^n(z)+\beta zp(z)p'(z)\prec e^z$ ($n =1,2,3$) implies $p(z)\prec e^z$.
\begin{theorem} \label{lem9}
Let $p\in \mathcal{H}_1$, then each of the following subordinations are sufficient for $p(z)\prec e^z$:
	\begin{enumerate}[(a)]
		\item $p(z)+\beta zp(z)p'(z) \prec e^z$ for $\beta \geq e^3+e  \approx 22.8038$.
		\item $p^2(z)+\beta z p(z)p'(z)\prec e^z$ for $\beta \geq e^3+1 \approx 21.0855$.
		\item $p^3(z)+\beta z p(z)p'(z)\prec e^z$ for $\beta \geq e^3+e^{-1} \approx 20.4534$.
	\end{enumerate}
\end{theorem}
\begin{proof} The subordination $p(z)\prec e^z$ is satisfied if we show that $\psi\in \Psi\{e^z\}$ for different choices of $\psi$. Equivalently, we need to verify the admissibility condition \eqref{adm}:
\[|\log\psi(r,s,t;z)|^2 \geq 1\]
whenever $r = e^{e^{i\theta}}$, $s = m e^{i\theta}r$ and $\real  \left(1+t/s\right)$ $\geq$ $m (1+\cos \theta )$, where $z\in \mathbb{D}$, $\theta \in [0,2\pi)$ and $m \geq 1$.

(a) Let $\psi(r,s,t;z)= r+\beta rs$. A simple calculation yields
 \begin{equation*} |\log \psi(r,s,t;z)|^2 = \frac{1}{4}\ln^2 (\mu^2+\nu^2)+ \left(\tan^{-1}(\chi) \right)^2 =: g(\theta)
\end{equation*}
where \[ \mu=e^{ \cos\theta} \sin(\sin\theta) +\beta m e^{2\cos\theta} \cos\theta\sin(2\sin\theta)+\beta m e^{2 \cos\theta}\sin\theta\cos(2\sin\theta)\] and \[\nu=e^{\cos\theta}\cos(\sin\theta) +\beta m e^{2\cos\theta} \cos\theta\cos(2\sin\theta)-\beta m e^{2\cos\theta}\sin\theta\sin(2\sin\theta). \]
It is easily verified that $g''(\pi)>0$ for  $\beta >\beta^* \approx 7.7065$, where $\beta^*$ is the positive	 root of the equation $6x-2e+(e-2x) \ln((e-x)^2)=0$. Since we are given that $\beta\geq e^3+e  \approx 22.8038$, the minimum value of $g$ is attained at $\theta =\pi$ and for all $\theta\in[0,2\pi)$, we have
\begin{align*}
g(\theta) \geq g(\pi) &=\frac{1}{4}\ln^2\left(\frac{1}{e^2}-\frac{2\beta m}{e^3}+\frac{\beta^2 m^2}{e^4}\right) \\
                      &=\frac{1}{4}\ln^2\left(\frac{\beta m}{e^2}-\frac{1}{e}\right)^2\geq \frac{1}{4}\ln^2\left(\frac{e^3+e}{e^2}-\frac{1}{e}\right)^2=1.
\end{align*}
Therefore $\psi \in \Psi\{e^z\}$.  
	
(b) The function $\psi(r,s,t;z)= r^2+\beta rs$ satisfies
	\begin{equation*} |\log \psi(r,s,t;z)|^2
	 = \frac{1}{4}\ln^2 \left(e^{4\cos\theta}(\mu^2+\nu^2)\right)+ \left(\tan^{-1} \frac{\mu}{\nu}\right)^2 =: g(\theta)
	\end{equation*}where
	\[ \mu=\sin(2\sin\theta)+\beta m \sin \theta\cos(2\sin\theta)+\beta m\cos\theta \sin(2\sin\theta)\]and \[\nu=\cos(2\sin\theta) +\beta m \cos \theta\cos(2\sin\theta)-\beta m\sin\theta \sin(2\sin\theta). \]
Since $g''(\pi)>0$ for  $\beta >\beta^* \approx 6.46722$, where $\beta^*$ is the root of the equation $x (2-3x)+ (2-3x+2x^2) \ln(-1 + x)=0$, the minimum value of $g(\theta)$ is obviously attained at $\theta =\pi$ if  $\beta \geq e^3+1 \approx 20.0855$. Therefore for $\theta\in[0,2\pi)$, we get \[g(\theta) \geq g(\pi) =\frac{1}{4}\ln^2\left(\frac{1}{e^4}-\frac{2\beta m}{e^4}+\frac{\beta^2 m^2}{e^4}\right) \geq 1\]  and hence $\psi \in \Psi\{e^z\}$. 
	
	(c) Let $\psi(r,s,t;z)= r^3+\beta rs$. With $r$, $s$ and $t$ stated above, $\psi$ takes the form $\psi(r,s,t;z)=e^{3e^{i\theta}}+\beta m e^{i\theta}e^{2e^{i\theta}}$ which satisfies
 \begin{equation*}	|\log \psi(r,s,t;z)|^2
	= \frac{1}{4}\ln^2 \left(e^{4\cos\theta}(\mu^2+\nu^2)\right)+ \left(\tan^{-1}\frac{\mu}{\nu}\right)^2 =:g(\theta)
	\end{equation*}
	where \[\mu=e^{\cos\theta} \sin(3\sin\theta)  +\beta m \cos\theta\sin(2\sin\theta)+\beta m \sin\theta\cos(2\sin\theta)\] and \[\nu=e^{\cos\theta}\cos(3\sin\theta)+\beta m \cos\theta\cos(2\sin\theta)-\beta m \sin\theta\sin(2\sin\theta). \]
Clearly, $g$ attains its minimum value either at $\theta=0$ or $\theta=\pi$. Since $g''(\pi)>0$  for $\beta >\beta^* \approx 5.66489$, where $\beta^*$ is the root of the equation $x e(3+5x e )-(3-x e+2x^2e^2) \ln(-1+ xe)=0$, the minimum value of $g$ is attained at $\theta =\pi$ if $\beta \geq e^3+e^{-1} \approx 20.4534$. In that case, we have \[g(\theta) \geq g(\pi) =\frac{1}{4}\ln^2\left(\frac{1}{e^6}-\frac{2\beta m}{e^5}+\frac{\beta^2 m^2}{e^4}\right)\geq 1\quad \text{for all }\theta\in[0,2\pi).\]
Therefore $\psi \in \Psi\{e^z\}$ by Theorem \ref{thm A}. 
\end{proof}
Now, we estimate the bound on $\beta$ such that the first order differential subordination $p^2(z)+p(z)-1+\beta zp'(z)\prec e^z$  implies $p(z)\prec e^z$.
\begin{theorem} \label{lem10}
	Let $p\in \mathcal{H}_1$ and satisfies the subordination
	\[p^2(z)+p(z)-1+\beta zp'(z)\prec e^z \quad \text{for }\beta \geq e^2+e^{-1}-e+1 \approx 6.03865. \] Then $p(z)\prec e^z$.
\end{theorem}
\begin{proof} Proceeding as in the previous theorems, we need to show that the function $\psi(r,s,t;z)  = r^2+r-1+\beta s$ satisfies the admissibility condition \eqref{adm}. Whenever $r = e^{e^{i\theta}}$, $s = m e^{i\theta}r$ and $\real  \left(1+t/s\right)$ $\geq$ $m (1+\cos \theta )$, where $z\in \mathbb{D}$, $\theta \in [0,2\pi)$ and $m \geq 1$,  $\psi(r,s,t;z)=e^{2e^{i\theta}}+e^{e^{i\theta}}-1+\beta m e^{i\theta}e^{e^{i\theta}}$ satisfies
\[|\log \psi(r,s,t;z)|^2= \frac{1}{4}\ln^2 (\mu^2  +\nu^2)+ \left(\tan^{-1}\frac{\nu}{\mu}\right)^2  =:g(\theta)\]
where
\[\mu=e^{2\cos\theta} \cos(2\sin\theta)+e^{\cos\theta} \cos(\sin\theta)+\beta m e^{\cos\theta} \cos \theta \cos(\sin\theta)-\beta m e^{\cos\theta} \sin\theta \sin(\sin\theta)-1\]and
\[\nu=e^{2\cos\theta} \sin(2\sin\theta)+e^{\cos\theta} \sin(\sin\theta)+\beta m e^{\cos\theta} \cos \theta \sin(\sin\theta)  +\beta m e^{\cos\theta} \sin\theta\cos(\sin\theta).\]

Using second derivative test, it can be easily verified that $g$ attains its minimum value at $\theta=\pi$ for $\beta>0$. Therefore for $\theta\in[0,2\pi)$, we have
\[g(\theta) \geq g(\pi) =\frac{1}{4}\ln^2\left(\frac{\beta m}{e}-\frac{1}{e^2}-\frac{1}{e}+1\right)^2.\]
Since logarithm is an increasing function and the condition $\beta \geq e^2+e^{-1}-e+1 $ imply that
\[|\log \psi(r,s,t;z)|^2 \geq \frac{1}{4}\ln^2\left(\frac{\beta }{e}-\frac{1}{e^2}-\frac{1}{e}+1\right)^2\geq \frac{1}{4}\ln^2(e^2)=1.\]
 Hence $\psi \in \Psi\{e^z\}$ and Theorem \ref{thm A} gives the desired result.
\end{proof}
\begin{remark}
As depicted in the previous section, results proved in this section also provide several sufficient conditions for a normalized analytic function to be in the class $\mathcal{S}^*_e$. These sufficient conditions can be obtained by simply putting $p(z)=zf'(z)/f(z)$, where $f\in\mathcal{H}$.
\end{remark}
\begin{remark}
Since we were concerned with the starlikeness property in this paper, therefore we presented applications of our results only for the subclasses of $\mathcal{S}^*$. However, by setting $p(z)=f'(z)$, $p(z)=f(z)/z$, $p(z)=2f(z)/z-1$, $p(z)=2\sqrt{f'(z)}-1$, $p(z)=2zf'(z)/f(z)-1$ and so forth in the theorems obtained, one can obtain many more differential implications.
\end{remark}
\bibliography{References}
\bibliographystyle{siam}
	
\end{document}